\documentclass{amsart}

\usepackage{amsmath,amsfonts,amsthm,amssymb,amscd,enumerate}
\usepackage{mathtools}
\mathtoolsset{centercolon}
\usepackage{mathrsfs}

\newtheorem{thm}{Theorem}[section]
\newtheorem{lem}[thm]{Lemma}
\newtheorem{cor}[thm]{Corollary}
\newtheorem{prop}[thm]{Proposition}
\theoremstyle{definition}
\newtheorem{dfn}[thm]{Definition}
\newtheorem{example}[thm]{Example}
\newtheorem{rmk}[thm]{Remark}

\numberwithin{equation}{section}

\usepackage[draft=false, colorlinks=true]{hyperref}

\title[Rescaled extrapolation for vector-valued functions]{Rescaled extrapolation for vector-valued functions}
\author{Alex Amenta, Emiel Lorist, and Mark Veraar}
\address{Delft Institute of Applied Mathematics \\ Delft University of Technology \\ P.O. Box 5031\\ 2600 GA Delft \\The Netherlands}
\thanks{The authors are supported by the VIDI subsidy 639.032.427 of the Netherlands Organisation for Scientific Research (NWO)}
\email{amenta@fastmail.fm}
\email{E.Lorist@tudelft.nl}
\email{M.C.Veraar@tudelft.nl}

\newcommand{\QQ}{\mathbb{Q}}
\newcommand{\CC}{\mathbb{C}}
\newcommand{\EE}{\mathbb{E}}
\newcommand{\NN}{\mathbb{N}}
\newcommand{\RR}{\mathbb{R}}

\newcommand{\TT}{\mathbb{T}}
\newcommand{\ZZ}{\mathbb{Z}}
\newcommand{\FF}{\mathscr{F}}

\newcommand{\Sch}{\mathcal{S}}
\newcommand{\calL}{\mc{L}}

\newcommand{\UMD}{\operatorname{UMD}}
\newcommand{\LPR}{\operatorname{LPR}}

\newcommand{\ind}{{{{\bf 1}}}}
\newcommand{\loc}{\operatorname{loc}}
\newcommand{\mb}{\mathbf}
\newcommand{\mc}{\mathcal}
\newcommand{\sgn}{\operatorname{sgn}}

\newcommand{\map}[3]{#1 \colon #2 \rightarrow #3}

\newcommand{\dd}{\hspace{2pt}\mathrm{d}}
\newcommand{\ee}{\mathrm{e}}

\newcommand{\wh}{\widehat}
\newcommand{\wt}{\widetilde}

\newcommand{\inc}{\mb{\phi}}   

\DeclarePairedDelimiter\abs{\lvert}{\rvert}

\DeclarePairedDelimiter\cbrace\{\}
\DeclarePairedDelimiter\ha()

\DeclarePairedDelimiter\nrm{\lVert}{\rVert}

\newcommand{\nrms}[1]{\Bigl\|#1\Bigr\|}
\newcommand{\abss}[1]{\Bigl|#1\Bigr|}
\newcommand{\has}[1]{\Bigl(#1\Bigr)}
\newcommand{\cbraces}[1]{\Bigl\{#1\Bigr\}}

\allowdisplaybreaks

\begin{document}

\begin{abstract}
  We extend Rubio de Francia's extrapolation theorem for functions valued in $\UMD$ Banach function spaces, leading to short proofs of some new and known results.
  In particular we prove Littlewood--Paley--Rubio de Francia-type estimates and boundedness of variational Carleson operators for Banach function spaces with $\UMD$ concavifications.
\end{abstract}

\keywords{extrapolation, Muckenhoupt weights, UMD, Banach function spaces, $p$-convexity, Hardy-Littlewood maximal function, Fourier multipliers, variational Carleson operator, Littlewood--Paley--Rubio de Francia inequalities}

\subjclass[2010]{Primary: 42B25; Secondary: 42A20, 42B15, 42B20, 46E30}


\maketitle

\section{Introduction}

The last few decades have seen many advances in the harmonic analysis of functions valued in a Banach space $X$.
Two cornerstone results are the boundedness of the lattice maximal function \cite{Bour:BCP, jR86}, and the equivalence of the $X$-valued Littlewood--Paley theorem and the $\UMD$ property for $X$, see \cite{Bou86}.
The Littlewood--Paley theorem is used to obtain extensions of the Marcinkiewicz multiplier theorem in \cite{Bou86} for scalar multipliers, and in \cite{We01} for operator-valued multipliers.
For an overview of these topics we refer to \cite{HNVW16}, and for useful applications to parabolic PDEs see for example \cite{DHP03,KW04,PruSim}.
Recent work on vector-valued harmonic analysis in UMD Banach function spaces includes \cite{BCCFR12, DK17, Fack14, sF16, HoMa16,  HMP08, HytVah, PSX12, Tag09, Xu15}.

In this paper we prove the following `rescaled' extrapolation theorem for $X$-valued functions (stated more precisely as Corollary \ref{cor:op-extrap}).
Here $\Sigma(\RR^d)$ denotes the simple functions $\RR^d \to \CC$, and $L^0(\RR^d)$ denotes the measurable functions $\RR^d \to \CC$ modulo almost everywhere equality.

\begin{thm}\label{thm:op-extrap-intro}
  Fix $p_0 \in (0,\infty)$.
  Suppose $T \colon \Sigma(\RR^d)\to L^0(\RR^d)$ satisfies
  \begin{equation}\label{eq:assTthmintro}
    |T(f) - T(g)|\leq |T(f-g)|, \qquad f,g\in \Sigma(\RR^d),
  \end{equation}
  and assume $T$ extends to a bounded operator on $L^p(\RR^d,w)$ for all $p > p_0$ and all Muckenhoupt weights $w \in A_{p/p_0}$. Let $X$ be a Banach function space and assume that for all $f \in \Sigma(\RR^d;X)$ the function $\map{\widetilde{T}f}{\RR^d}{X}$, defined by
  \begin{equation*}
    \widetilde{T}f(x,\omega):= \bigl(Tf(\cdot,\omega)\bigr)(x), \qquad x \in \RR^d, \quad \omega \in \Omega,
  \end{equation*}
  is well-defined and strongly measurable.
  If $X$ is $p_0$-convex and $X^{p_0}$ has the $\UMD$ property, then $\widetilde{T}$ extends to a bounded operator on
   $L^p(\RR^d,w;X)$ for all $p \in (p_0,\infty)$ and $w \in A_{p/p_0}$.
\end{thm}

The assumption \eqref{eq:assTthmintro} holds in particular if $T$ is a linear operator, or if $T$ is a sublinear operator such that $Tf \geq 0$ for all $f \in \Sigma(\RR^d)$. In applications it is usually easy to check that $\widetilde{T}$ is well-defined and strongly measurable; see for example the operators in Sections \ref{sec:vcarl} and \ref{sec:LPR}. If $T$ is linear, then the extension coincides with the standard tensor extension, which is automatically well-defined and strongly measurable.

For $p_0=1$, and with $\RR^d$ replaced by the torus $\TT$, this result is proved in \cite[Theorem 5]{jR86}.
The main ingredient in the proof is the boundedness of the lattice maximal operator (see Theorem \ref{thm:maximalfunctionweighted}).
In fact, we deduce Theorem \ref{thm:op-extrap-intro} from a more general extrapolation theorem for pairs of functions (Theorem \ref{thm:pair-extrap-p}).
Further details may be found in Section \ref{sec:extrapolation}.

We use Theorem \ref{thm:op-extrap-intro} to prove two important results: vector-valued Littlewood--Paley--Rubio de Francia-type estimates (Section \ref{sec:LPR}), and boundedness of vector-valued variational Carleson operators (Section \ref{sec:vcarl}).
We also establish the boundedness of some scalar-valued Fourier multipliers on vector-valued functions (Section \ref{sec:applications}); we will obtain deeper operator-valued multiplier results from vector-valued Littlewood--Paley--Rubio de Francia-type estimates in \cite{ALV2}.

Our main motivation for this paper are the vector-valued Littlewood--Paley--Rubio de Francia-type estimates, which we briefly explain.
For an interval $I \subset \RR$, let $S_I$ denote the Fourier projection onto $I$, defined by $S_I f := \FF^{-1} (\mathbf{1}_I \hat{f} )$ for Schwartz functions $f$ on the real line.
For every collection $\mc{I}$ of pairwise disjoint intervals and every $q \in (0,\infty]$ we consider the operator
\begin{equation*}
  \mc{S}_{\mc{I},q}(f) := \big( \sum_{I \in \mc{I}} |S_I f|^q \big)^{1/q},
\end{equation*}
interpreted as a supremum when $q = \infty$.
If $\mc{I}$ is a dyadic decomposition of $\RR$, then the classical Littlewood--Paley inequality states that $\|\mc{S}_{\mc{I},2} f\|_{L^p} \eqsim \|f\|_{L^p}$ for $p \in (1,\infty)$.
In \cite{jR85} Rubio de Francia proves the $L^p$-boundedness of $\mc{S}_{\mc{I},q}$ when $\mc{I}$ is an \emph{arbitrary} collection of disjoint intervals, $q \in [2,\infty]$, and $p \in (q^\prime,\infty)$; this result (particularly the $q=2$ case) is now known as the Littlewood--Paley--Rubio de Francia theorem.

The definition of $S_I$ extends directly to the vector-valued setting.
Vector-valued extensions of the Littlewood--Paley--Rubio de Francia theorem for the case $q=2$ case are studied in \cite{BGT03,GilTor04,HP06,HTY09,PSX12} via a reformulation in terms of random sums,
\begin{equation*}
	\EE \nrms{ \sum_{I \in \mc{I}} \varepsilon_I S_I f }_{L^p(\RR;X)} \lesssim \nrm{f}_{L^p(\RR;X)},
\end{equation*}
where $(\varepsilon_I)_{I \in \mc{I}}$ is a sequence of independent Rademacher variables and $\EE$ denotes the expectation.
If this estimate holds then we say that $X$ has the $\LPR_{p,2}$ property, or in short, that $X$ is $\LPR_{p,2}$.
When $X$ is a $\UMD$ Banach function space, this is equivalent to the boundedness of $\mc{S}_{\mc{I},2}$ on $L^p(\RR;X)$.
However, when $q \neq 2$ no analogue of the boundedness of $\mc{S}_{\mc{I},q}$ for general Banach spaces is known.

The $\LPR_{p,2}$ property is quite mysterious.
In \cite[Theorem 1.2]{HTY09} it was shown that if a Banach space $X$ is $\LPR_{p,2}$ for some $p \geq 2$, then $X$ is $\UMD$ and has type $2$.
However, the converse is only known to hold when the collection $\mc{I}$ consists of intervals of equal length.
The most general sufficient condition currently known is in \cite[Theorem 3]{PSX12}: if $X$ is a $2$-convex Banach lattice and the $2$-concavification $X^2$ is $\UMD$, then $X$ is $\LPR_{p,2}$ for all $p > 2$.
This result is proved by an extension of Rubio de Francia's argument for the scalar-valued case.
Every Banach space $X$ that is known to have the $\LPR_{p,2}$ property is either of this form, or is isomorphic to a Hilbert space (and hence is $\LPR_{p,2}$ for all $p \geq 2$, by Rubio de Francia's original proof).

We prove the following theorem (a more precise version of which appears as Theorem \ref{thm:LPR-main}).

\begin{thm}\label{thm:lpr-main-intro}
  Let $q \in [2,\infty)$,
  and suppose $X$ is a $q$-convex Banach function space whose $q'$-concavification $X^{q'}$ is $\UMD$.
  Then there exists a nondecreasing function $\map{\inc_{X,p,q}}{[1,\infty)}{[1,\infty)}$ such that
  \begin{equation*}
  	\|\mc{S}_{\mc{I},q} f\|_{L^p(w;X)} \leq \inc_{X,p,q}([w]_{A_{p/q'}}) \nrm{f}_{L^p(w;X)}
	\end{equation*}
 for all $p\in (q',\infty)$, all Muckenhoupt weights $w \in A_{p/q'}$, and all $f \in L^p(w;X)$.
\end{thm}

We deduce this result, which includes \cite[Theorem 3]{PSX12} as a special case, directly from the scalar case $X = \CC$ via Theorem \ref{thm:op-extrap-intro}.
See Section \ref{sec:LPR} for further details.

\subsection*{Notation}

If $\Omega$ is a measure space (we omit reference to the measure unless it is needed) and $X$ is a Banach space, we let $\Sigma(\Omega;X)$ denote the vector space of simple functions $\Omega \to X$, and $L^0(\Omega;X)$ denote the vector space of strongly measurable functions modulo almost-everywhere equality.
When $X = \CC$ we denote these sets by $\Sigma(\Omega)$ and $L^0(\Omega)$.
When $X$ is a Banach function space we let $L^0_+(\Omega;X)$ denote the space of (almost everywhere) non-negative functions in $L^0(\Omega;X)$.
For Banach spaces $X$ and $Y$, $\mc{B}(X,Y)$ denotes the bounded operators and $\mc{L}(X, Y)$ the bounded linear operators from $X$ into $Y$.

Throughout the paper we write $\inc_{a,b,\ldots}$ to denote a non-decreasing function $[1,\infty) \to [1,\infty)$ which depends only on the parameters $a,b,\ldots$, and which may change from line to line.
Non-decreasing dependence on the Muckenhoupt characteristic of weights is needed for extrapolation theorems.
We do not obtain sharp dependence on Muckenhoupt characteristics in our results, but we need to be careful in tracking monotonicity of estimates in these characteristics.
In Appendix \ref{sec:weight dependence} we show that monotone dependence on the Muckenhoupt characteristic can be deduced from a more general estimate in terms of the characteristic.

Occasionally we will work with $\RR^d$ for a fixed dimension $d \geq 1$.
Implicit constants in estimates will depend on $d$, but we will not state this.

\subsection*{Acknowledgements}
We thank Gennady Uraltsev for bringing the results of \cite{dPDU16} and \cite{OSTTW12} to our attention, Sebastian Kr\'ol for interesting discussions on extrapolation and the anonymous referee for their helpful comments.

\section{Preliminaries\label{sec:prel}}

\subsection{Banach function spaces\label{subs:Bfs}}

\begin{dfn}\label{def:bfs}
  Let $\Omega$ be a measure space.
  A subspace $X$ of $L^0(\Omega)$ equipped with a norm $\nrm{\cdot}_X$ is called a {\em Banach function space} (over $\Omega$) if it satisfies the following properties:
  \begin{enumerate}[(i)]
    \item If $x \in L^0(\Omega), y \in X$, and $\abs{x} \leq \abs{y}$, then $x \in X$ and $\nrm{x}_X \leq \nrm{y}_X$.
    \item There exists $\zeta \in X$ with $\zeta>0$.
    \item If $0 \leq x_n \uparrow x$ with $(x_n)_{n=1}^\infty$ a sequence in $X$, $x \in L^0(\Omega)$, and $\sup_{n \in \NN} \nrm{x_n}_X < \infty$, then $x \in X$ and $\nrm{x}_X = \sup_{n \in \NN}\nrm{x_n}_X$.
  \end{enumerate}
  A Banach function space $X$ is {\em order continuous} if for any $0 \leq x_n \uparrow x$ with $(x_n)_{n=1}^\infty$ a sequence in $X$ and $x \in X$, we have $\nrm{x - x_n}_X \to 0$.
\end{dfn}

\begin{dfn}
  Let $X$ be a Banach function space and $p \in [1,\infty]$.
  We say that $X$ is {\em $p$-convex} if
  \begin{equation*}
    \nrms{\has{\sum_{k=1}^n \abs{x_k}^p}^\frac{1}{p}}_X \leq \has{\sum_{k=1}^n\nrm{x_k}_X^p}^\frac{1}{p}
  \end{equation*}
  for all $x_1,\cdots,x_n \in X$, with the usual modification when $p = \infty$.
  We say that $X$ is {\em $p$-concave} if the reverse estimate holds.
\end{dfn}

Every Banach function space is $1$-convex and $\infty$-concave, and furthermore if a Banach function space is $p$-convex and $q$-concave then $p \leq q$.
As a simple example, we note that $L^r$ is $p$-convex for all $p \in [1,r]$ and $q$-concave for all $q \in [r,\infty]$.
The definitions of $p$-convexity and $p$-concavity usually include an implicit constant depending on $p$ and $X$, but if such an estimate holds then $X$ may be equivalently renormed so that these constants are equal to $1$ (see \cite[Theorem 1.d.8]{LT79}).
Since our results are stable under equivalence of norms, we may consider the stronger definition above without loss of generality.

The following elementary properties are proved in \cite[Section 1.d]{LT79}.

\begin{prop}
\label{prop:convexconcaveinterval}
  Let $X$ be a Banach function space and $p_0 \in [1,\infty]$.
  \begin{enumerate}[(i)]
    \item If $X$ is $p_0$-convex, then $X$ is $p$-convex for all $p \in [1,p_0]$.
    \item If $X$ is $p_0$-concave, then $X$ is $p$-concave for all $p \in [p_0,\infty]$.
    \item $X$ is $p_0$-convex if and only if $X^*$ is $p_0^\prime$-concave.
  \end{enumerate}
\end{prop}

Let $X$ be a Banach function space over a measure space $\Omega$, and let $s \in (0,\infty)$.
We define the {\em $s$-concavification} $X^s$ of $X$ by
\begin{equation}\label{eq:concavification}
  X^{s} := \cbrace*{x \in L^0(\Omega): \abs{x}^{1/s} \in X} = \cbrace*{\abs{x}^s\sgn(x): x \in X},
\end{equation}
where $\sgn$ is the complex signum function, endowed with the quasinorm
\begin{equation*}
  \nrm{x}_{X^s} := \big\| |x|^{1/s}\big\|_X^{s}.
\end{equation*}
By Proposition \ref{prop:convexconcaveinterval}, when $s > 1$, $X^s$ is a Banach space if and only if $X$ is $p$-convex for some $p \geq s$.
On the other hand, when $s \leq 1$, $X^s$ is always a Banach space.
As a key example, for $0 < r \leq p < \infty$ the $r$-concavification of $L^p$ is $(L^p)^r = L^{p/r}$.

The following simple density lemma will be applied several times.
It is not difficult---some may consider it obvious---but it should be emphasised.

\begin{lem}\label{lem:densityextension}
Assume $T \colon \Sigma(\RR^d)\to L^0(\RR^d)$ satisfies
\begin{equation}\label{eq:assTlem}
|T(f) - T(g)|\leq |T(f-g)|, \qquad f,g\in \Sigma(\RR^d).
\end{equation}
Let $X$ be a Banach function space over $(\Omega,\mu)$ and assume that for all $f \in \Sigma(\RR^d;X)$ the function $\map{\widetilde{T}f}{\RR^d}{X}$, defined by
  \begin{equation*}
    \widetilde{T}f(x,\omega):= \bigl(Tf(\cdot,\omega)\bigr)(x), \qquad x \in \RR^d, \quad \omega \in \Omega,
  \end{equation*}
  is well-defined and strongly measurable. Let $w \colon \RR^d\to (0,\infty)$ be a locally integrable function, and $p\in (0, \infty)$.
If there exists a constant $C\geq 0$ such that
\begin{equation*}
  \|\widetilde{T}(f)\|_{L^p(w;X)}\leq C\|f\|_{L^p(w;X)}, \qquad f\in \Sigma(\RR^d;X),
\end{equation*}
then $\widetilde{T}$ extends to a bounded operator on $L^p(w;X)$.
\end{lem}

Note that \eqref{eq:assTlem} holds for all linear operators $T \colon \Sigma(\RR^d)\to L^0(\RR^d)$ and for all positively-valued sublinear operators $T \colon \Sigma(\RR^d)\to L^0_+(\RR^d)$ (such as maximal functions or square functions).

\begin{proof}
  For all $f,g\in \Sigma(\RR^d;X)$ we have $|\widetilde{T}(f) - \widetilde{T}(g)|\leq |\widetilde{T}(f-g)|$ pointwise in $\Omega$, so it follows that
  \begin{equation*}
    \|\widetilde{T}(f) - \widetilde{T}(g)\|_{L^p(w;X)}\leq \|\widetilde{T}(f-g)\|_{L^p(w;X)}\leq C \|f-g\|_{L^p(w;X)}.
  \end{equation*}
Therefore $\widetilde{T}$ is Lipschitz continuous, and thus uniquely extends to a bounded operator on $L^p(w;X)$ by density of $\Sigma(\RR^d;X)$ in $L^p(w;X)$.
\end{proof}

\begin{rmk}
Although our results are stated in terms of Banach function spaces, many of them extend to spaces which are isomorphic to a closed subspace of a Banach function space, and by standard representation techniques many results extend to Banach lattices.
We refer to \cite{LT79, MeyNie} for details.
\end{rmk}

\subsection{Muckenhoupt weights\label{subs:muck}}
A \emph{weight} on $\RR^d$ is a nonnegative function $w \in L^1_{\loc}(\RR^d)$.
For $p \in [1,\infty)$ the space $L^p(w) = L^p(\RR^d,w)$ is the subspace of all $f \in L^0(\RR^d)$ such that
\begin{equation*}
  \nrm{f}_{L^p(w)} := \has{\int_{\RR^d}\abs{f(x)}^p w(x) \dd x}^{1/p}< \infty.
\end{equation*}
The \emph{Muckenhoupt $A_p$ class} is the set of all weights $w$ such that
  \begin{equation*}
    [w]_{A_{p}} :=\sup_{B} \frac{1}{\abs{B}} \int_B w(x) \dd x \cdot \has{\frac{1}{\abs{B}} \int_B w(x)^{-\frac{1}{p-1}}\dd x}^{p-1} < \infty,
  \end{equation*}
  where the supremum is taken over all balls $B \subset \RR^d$, and where the second factor is replaced by $\nrm{w^{-1}}_{L^\infty(B)}$ when $p=1$.
  We define $A_\infty = \bigcup_{p \geq 1} A_p$.
  When $p \in (1,\infty)$, a weight $w$ is in $A_p$ if and only if the Hardy-Littlewood maximal operator $M$ is bounded on $L^p(w)$; this operator is defined on $f \in L^1_{\loc}(\RR^d)$ by
\begin{equation}\label{eq:Maximalf}
  Mf(x) := \sup_{r>0}\frac{1}{\abs{B(x,r)}}\int_{B(x,r)} \abs{f(y)} \dd y, \qquad x \in \RR^d.
\end{equation}

Proofs of the following properties can be found in \cite[Chapter 9]{lG09}.

\begin{prop}\label{prop:muckenhoupt}\
\begin{enumerate}[(i)]
   \item \label{it:mw3} The $A_p$ classes are increasing in $p$, with $[w]_{A_q} \geq [w]_{A_p}$ when $1 \leq q \leq p$.
  \item \label{it:mw5} For all $w \in A_p$ with $p \in (1,\infty)$ there exists $\varepsilon>0$ such that $w \in A_{p-\varepsilon}$.
  \item \label{it:mw4} For all $p \in (1,\infty)$ and all weights $w$,
  \begin{equation*}
    \nrm{M}_{\mc{B}(L^p(w))} \lesssim [w]_{A_p}^{\frac{1}{p-1}} \lesssim \nrm{M}_{\mc{B}(L^p(w))}^{p^\prime}
  \end{equation*}
  with implicit constants independent of $w$.
\end{enumerate}
\end{prop}

These definitions could be made in terms of cubes with sides parallel to the coordinate axes instead of balls.
This results in equivalent definitions up to dimensional constants.
Moreover one could replace the measure on $\RR^d$ with a general doubling measure. For further details on Muckenhoupt weights see \cite{CMP11} and \cite[Chapter 9]{lG09}.

\subsection{\texorpdfstring{The $\UMD$ property}{The UMD property}\label{subs:UMD}}

A Banach space $X$ has the $\UMD$ property if and only if the Hilbert transform extends to a bounded operator on $L^p(\RR;X)$.
This is a major result of Burkholder \cite{Burk83} and Bourgain \cite{Bour83}, and it also makes for a convenient definition.
For a detailed account of the theory of $\UMD$ spaces we refer the reader to \cite{Burk01} and \cite{HNVW16}.
The ``classical'' reflexive spaces---$L^p$ spaces, Sobolev spaces, Triebel--Lizorkin and Besov spaces, Schatten clases, among others---have the $\UMD$ property.
However, the $\UMD$ property implies reflexivity, so $L^1$ and $L^\infty$ (in particular) are not $\UMD$.

The theory of $\UMD$ Banach function spaces is very rich, and we refer to \cite{jR86} for an overview.
A connection between the $\UMD$ property and convexity is given by the following result, which is proved by combining \cite[Proposition 4.2.19]{HNVW16}, \cite[Theorem 11.1.14]{AK06}, and \cite[Corollary 1.f.9]{LT79}

\begin{prop}\label{prop:UMDtype}
  Let $X$ be a $\UMD$ Banach function space. Then $X$ is $p$-convex and $q$-concave for some $1<p<q<\infty$.
\end{prop}

A connection between the $\UMD$ property and the Hardy--Littlewood maximal operator is provided via $L^p(\RR^d;X)$-boundedness of the \emph{lattice maximal operator} $\wt{M}$.
Let $X$ be a Banach function space over a measure space $\Omega$.
For all simple functions $f\in\Sigma(\RR^d;X)$ let
\begin{equation*}
  \wt{M} f(x,\omega) = M(f(\cdot,\omega))(x), \qquad (x,\omega) \in \RR^d \times  \Omega,
\end{equation*}
where $M$ is the Hardy--Littlewood maximal operator as defined in \eqref{eq:Maximalf}. Recall that $\inc_{a,b,\ldots}$ denotes an unspecified nondecreasing function $[1,\infty) \to [1,\infty)$ which depends only on the parameters $a,b,\ldots$, and which may change from line to line.

\begin{thm}\label{thm:maximalfunctionweighted}
  Suppose $X$ is a $\UMD$ Banach function space, $p \in (1,\infty)$, and $w\in A_p$.
  Then $\wt{M}$ is bounded on $L^p(w;X)$, and
  \begin{equation*}
    \|\widetilde{M}\|_{\mc{B}(L^p(w;X))} \leq  \inc_{X,p}([w]_{A_p}).
  \end{equation*}
\end{thm}

Note that $\widetilde{M}$ was initially defined on $\Sigma(\RR^d;X)$, but can now be extended to $L^p(w;X)$ by density and boundedness (see Lemma \ref{lem:densityextension}).
A converse to Theorem \ref{thm:maximalfunctionweighted} also holds: if $\wt{M}$ is bounded on both $L^p(\RR^d;X)$ and $L^p(\RR^d;X^*)$, then $X$ is $\UMD$.
The unweighted case of Theorem \ref{thm:maximalfunctionweighted} on the torus is proved in \cite{Bour:BCP} and \cite{jR86} and the weighted case on $\RR^d$ in \cite{GCMT93} (see \cite[Theorem 5.6.4]{eL16} for more precise dependence on $[w]_{A_p}$).

We often consider $s$-convex Banach function spaces $X$ such that $X^s$ is $\UMD$.
This condition is open in $s$: if $X^s$ is $\UMD$, then there exists $\varepsilon > 0$ such that $X^r$ is $\UMD$ for all $0 < r < s+\varepsilon$ \cite[Theorem 4]{jR86}.
In particular if $X^s$ is $\UMD$ for some $s \geq 1$, then $X$ is $\UMD$, and conversely if $X$ is $\UMD$ then $X^s$ is $\UMD$ for some $s > 1$.

\begin{rmk}
  Throughout the paper we will write `$X^s \in \UMD$' as a shortcut for `$X^s$ is a Banach space with the $\UMD$ property'.
  If $s \geq 1$ this therefore implies that $X$ is $s$-convex.
\end{rmk}

\section{Extrapolation}\label{sec:extrapolation}
One of the most important features of the Muckenhoupt classes is the celebrated Rubio de Francia extrapolation theorem (see \cite{jR82,jR83,jR84,Rub87} and \cite[Chapter IV]{GR85}).
This allows one to deduce estimates for \textit{all} $p\in (1,\infty)$ and all $w \in A_p$ from the corresponding estimates for a \emph{single} $p \in (1,\infty)$ and all $w \in A_{p}$.
A rescaled version of the theorem can be formulated as follows; see \cite[Theorems 3.9 and Corollary 3.14]{CMP11} for a simple proof.
Recall our convention that $\inc_{a,b,\cdots}$ denotes an unspecified nondecreasing function $[1,\infty) \to [1,\infty)$ which depends only on the parameters $a,b,\cdots$ and which may change from line to line.

\begin{thm}[Scalar-valued rescaled extrapolation]\label{thm:extrapolation}
  Fix $p_0\in (0,\infty)$. Suppose that $\mc{F} \subset L_+^0(\RR^d) \times L_+^0(\RR^d)$ and that for all $(f,g) \in \mc{F}$, some $p\in(p_0,\infty)$, and all $w \in A_{p/p_0}$, the estimate
\begin{equation*}
  \nrm{ f}_{L^{p}(w)} \leq \inc_{p,p_0}([w]_{A_{p/p_0}}) \nrm{ g}_{L^{p}(w)}
\end{equation*}
holds.
Then the same estimate holds for all $(f,g) \in \mc{F}$, $p \in (p_0,\infty)$, and $w \in A_{p/p_0}$.
\end{thm}

In this section we prove the following vector-valued extrapolation theorem, which extends another of Rubio de Francia's extrapolation theorems \cite[Theorem 5]{jR86}.

\begin{thm}[Vector-valued rescaled extrapolation]\label{thm:pair-extrap-p}
	Fix $p_0 \in (0,\infty)$ and suppose that $X$ is a Banach function space over a measure space $(\Omega,\mu)$ with $X^{p_0}\in \UMD$.
Suppose that $\mc{F} \subset L^0_+(\RR^d;X) \times L^0_+(\RR^d;X)$ and that for all $p > p_0$, $(f,g) \in \mc{F}$, and $w \in A_{p/p_0}$, we have
  \begin{equation}\label{eqn:extrap-assn}
    \nrm{f(\cdot,\omega)}_{L^p(w)} \leq \inc_{p,p_0} ([w]_{A_{p/p_0}}) \nrm{g(\cdot,\omega)}_{L^p(w)}, \qquad \mu\text{-a.e. }\omega \in \Omega.
  \end{equation}
  Then for all $p>p_0$, $(f,g) \in \mc{F}$, and $w \in A_{p/p_0}$, we have
  \begin{equation}\label{eqn:extrap-goal}
    \nrm{f}_{L^p(w;X)} \leq \inc_{X,p,p_0} ( [w]_{A_{p/p_0}} ) \nrm{g}_{L^p(w;X)}.
  \end{equation}
\end{thm}

\begin{rmk}
By Theorem \ref{thm:extrapolation} it suffices to have \eqref{eqn:extrap-assn} for some $p \in (p_0,\infty)$ and all $w \in A_{p/p_0}$.
\end{rmk}

To prove Theorem \ref{thm:pair-extrap-p} we need some preliminary lemmas.
The first is a combination of \cite[Lemma 1, p. 217]{jR86} and \cite[Corollary 5.3]{GR85}.
We include the proof for the reader's convenience.
The second is a modification of \cite[Lemma 2, p. 218]{jR86}.
We emphasise that the operators need not be linear, and that if $Y$ is $\UMD$, then it is reflexive and thus order continuous (see \cite[Section 2.4]{MeyNie}).

\begin{lem}\label{lem:RdF1}
  Suppose $q \in [1,\infty)$.
  Let $Y$ be a $q$-convex order continuous Banach function space over a measure space $(\Omega,\mu)$, and let $\Gamma\subset \mc{B}(Y)$ be a set of bounded operators such that $\abs{T(\lambda y)} = \lambda \,\abs{T(y)}$ for all $y \in Y$, $T \in \Gamma$ and $\lambda>0$.
  Then the following are equivalent:
  \begin{enumerate}[(i)]
  \item There exists $C_1>0$ such that for all $T_1,\cdots,T_n \in \Gamma$ and $y_1,\cdots y_n \in Y$,
    \begin{equation}
      \label{eq:RdF1}
      \nrms{\has{\sum_{k=1}^n\abs{T_ky_k}^q}^\frac{1}{q}}_Y \leq C_1 \nrms{\has{\sum_{k=1}^n \abs{y_k}^q}^\frac{1}{q}}_Y.
    \end{equation}
  \item There exists $C_2>0$ such that for every nonnegative $u \in (Y^q)^*$, there exists $v \in (Y^q)^*$ with $u \leq v$, $\nrm{v} \leq 2\nrm{u}$, and
    \begin{equation}
      \label{eq:RdF2}
      \int_\Omega |Ty|^q v \dd\mu \leq C_2 \, \int_\Omega |y|^q v \dd\mu, \qquad y \in Y, \quad T \in \Gamma.
    \end{equation}
  \end{enumerate}
  Moreover $C_1 \eqsim C_2$.
\end{lem}

Note that $(Y^{q})^*$ is a Banach function since $Y$ is order continuous (see \cite[Section 1.b]{LT79}), so \eqref{eq:RdF2} is well-defined.

\begin{proof}
  We first prove (ii) $\Rightarrow$ (i). Let $y_1,\cdots,y_n \in Y$ and $T_1,\cdots,T_n \in \Gamma$.  Since $\sum_{k=1}^n \abs{T_ky_k}^q \in Y^{q}$ we can find a nonnegative $u \in (Y^{q})^*$ with $\nrm{u} =1$ such that
\begin{equation*}
  \nrms{\has{\sum_{k=1}^n\abs{T_ky_k}^q}^{\frac{1}{q}}}^q_{Y} = \nrms{\has{\sum_{k=1}^n\abs{T_ky_k}^q}}_{Y^{q}} = \int_{\Omega} \sum_{k=1}^n \abs{T_ky_k}^qu \dd\mu.
\end{equation*}
Then by assumption there exists $u \geq v$ with $\nrm{v} \leq 2$ and
\begin{equation*}
  \has{\int_{\Omega} \sum_{k=1}^n \abs{T_ky_k}^qu \dd\mu}^{\frac{1}{q}}
  \leq C_2 \has{\int_{\Omega} \sum_{k=1}^n \abs{y_k}^qv \dd\mu}^\frac{1}{q}
  \leq 2^{\frac{1}{q}} C_2  \nrms{\has{\sum_{k=1}^n \abs{y_k}^q}^\frac{1}{q}}_Y,
\end{equation*}
which proves \eqref{eq:RdF1}.

Now for (i) $\Rightarrow$ (ii) take a nonnegative $u \in (Y^q)^*$. Without loss of generality we may assume that $\nrm{u}\leq1$.
Let $Z := L^q(\Omega,u)$.
Then $\nrm{y}_Z \leq \nrm{y}_Y$ for all $y \in Y$, so $Y \hookrightarrow Z$.
We can therefore consider $\Gamma$ as a family of operators from $Y$ to $Z$ with
\begin{equation}\label{eq:RdF3}
  \has{\sum_{k=1}^n \nrm{T_ky_k}_Z^q}^\frac{1}{q}
  \leq \nrms{\has{\sum_{k=1}^n\abs{T_ky_k}^q}}_{Y^q}^\frac{1}{q}
  \leq C_1\nrms{\has{\sum_{k=1}^n \abs{y_k}^q}^\frac{1}{q}}_Y
\end{equation}
for all $y_1,\cdots y_n \in Y$ and $T_1,\cdots,T_n \in \Gamma$ by \eqref{eq:RdF1}.

Define the sets
\begin{align*}
  A &:= \cbraces{\has{\sum_{k=1}^n \abs{y_k}^q, \sum_{k=1}^n \nrm{T_ky_k}_Z^q}:  y_k \in Y , T_k \in \Gamma} \subset Y^{q} \times \RR,\\
  B &:= \cbraces{b \in (Y^{q})^*: \nrm{b} \leq 1 \text{ and } b \geq 0}.
\end{align*}
Since $\abs{T(\lambda y)} = \lambda\,\abs{T(y)}$ for all $y \in Y$, $T \in \Gamma$ and $0<\lambda<1$ we see that $A$ is convex.
The set $B$ is also convex, and by the Banach-Alaoglu theorem $B$ is weak$^*$-compact.

Define $\Phi \colon A \times B \to \RR$ by
\begin{equation*}
  \Phi(a,b) := \sum_{k=1}^n\nrm{T_ky_k}_Z^q -C_1^q \int_\Omega \sum_{k=1}^n\abs{y_k}^q \, b \dd\mu,  \qquad a=\has{\sum_{k=1}^n \abs{y_k}^q, \sum_{k=1}^n \nrm{T_ky_k}_Z^q}.
\end{equation*}
Then $\Phi$ is linear in its first coordinate and affine in its second.
Furthermore, by definition $\Phi(a,\cdot)$ is weak$^*$-continuous for all $a \in A$, and by \eqref{eq:RdF3} for any $a \in A$
\begin{align*}
  \min_{b \in B} \Phi(a,b)
  =  \sum_{k=1}^n\nrm{T_ky_k}_Z^q- C^q_1\nrms{\has{\sum_{k=1}^n\abs{y_k}^q}^\frac{1}{q}}^q_{Y} \leq 0.
\end{align*}
Thus by the Minimax lemma (see \cite[Appendix H]{lG08}),
\begin{equation*}
  \adjustlimits\min_{b \in B} \sup_{a \in A} \Phi(a,b)= \adjustlimits\sup_{a \in A} \min_{b \in B}  \Phi(a,b)  \leq 0,
\end{equation*}
so there exists $w_1 \in B$ such that $\Phi(a,w_1) \leq 0$ for all $a \in A$.
In particular, for any $y \in Y$ and $T \in \Gamma$ we find that
\begin{equation*}
 \int_\Omega |Ty|^q u \, d\mu  -C_1^q \int_\Omega \abs{y}^q w_1\dd \mu = \Phi\bigl((\abs{y}^q, \nrm{Ty}_Z^q),w_1\bigr) \leq 0.
\end{equation*}

Set $w_0 := u$.
Iterating the argument with $w_1$ in place of $u$ yields a sequence $(w_n)_{n=1}^\infty$ satisfying
\begin{equation*}
  \has{\int_\Omega\abs{Ty}^qw_n \dd\mu}^\frac{1}{q} \leq C_1 \has{\int_\Omega\abs{y}^qw_{n+1}  \dd\mu}^\frac{1}{q}, \qquad y \in Y, \quad T \in \Gamma
\end{equation*}
for all $n \in \NN$.
Then $v := \sum_{n=0}^\infty 2^{-n} w_n$ satisfies $u \leq v$, $\nrm{v} \leq 2$ and \eqref{eq:RdF2}.
\end{proof}

\begin{lem}\label{lem:RdF-algo}
  Suppose that $1<q<p<\infty$ and let $X$ be a  $q$-convex  $\UMD$ Banach function space over a measure space $(\Omega,\mu)$.
  Then for all $w \in A_p$ and every nonnegative $u \in L^{(p/q)^\prime}(w;(X^q)^*)$, there exists $v \in L^{(p/q)^\prime}(w;(X^q)^*)$ such that $u \leq v$, $\nrm{v} \leq 2\nrm{u}$, and $v(\cdot,\omega)w \in A_q$ for $\mu$-almost every $\omega \in \Omega$.
  Moreover,
  \begin{equation}\label{eqn:vw-Aq}
    [v(\cdot,\omega)w]_{A_q} \leq \inc_{X,p,q}([w]_{A_p}), \qquad \text{$\mu$-a.e. $\omega \in \Omega$.}
  \end{equation}
\end{lem}

\begin{proof}
  Suppose $w \in A_p$ and $u \in L^{(p/q)^\prime}(w;(X^q)^*)$.
  By \cite[p.\ 214]{jR86}, $X(\ell^q)$ has the UMD property.
  Thus, by Theorem \ref{thm:maximalfunctionweighted} the lattice maximal operator $\widetilde{M}$ satisfies \eqref{eq:RdF1} for $Y = L^p(w;X)$, with constant $\inc_{X,p,q}([w]_{A_p})$.
  Note that $Y$ is $q$-convex since $q < p$ and by \cite[Theorem 1.3.10]{HNVW16},
  \begin{equation*}
    (Y^q)^* = (L^{(p/q)}(w;X^q))^* = L^{(p/q)^\prime}(w;(X^q)^*),
  \end{equation*}
  using $w\dd x$ as the measure on $\RR^d$ in the second equality.

  Applying Lemma \ref{lem:RdF1} to $Y$ with $T=\widetilde{M}$, we deduce that there exists $v \in L^{(p/q)^\prime}(w;(X^q)^*)$ with $u \leq v$, $\nrm{v} \leq 2\nrm{u}$, and
  \begin{equation}\label{eqn:v-extraction}
    \int_{\RR^d} \int_{\Omega} |\widetilde{M}f|^q  v  \dd\mu \, w \dd x  \leq \inc_{X,p,q}([w]_{A_p}) \int_{\RR^d} \int_{\Omega} |f|^q   v  \dd \mu \, w \dd x
  \end{equation}
  for all $f \in L^p(w;X)$. Now let
  \begin{align*}
    \mc{A} &= \cbraces{\sum_{j=1}^n a_j\ind_{Q_j}: a_j  \in \QQ \oplus i\QQ \text{ and } Q_j\subset \RR^d \text{ rectangles with rational endpoints}},
  \intertext{fix $\zeta \in X$ with $\zeta>0$ and define}
    \mc{B} &= \cbrace{f: f(x,\omega) = \varphi(x)\ind_E(\omega)\zeta(\omega) \text{ with } \varphi \in \mc{A}, E \subseteq \Omega \text{ measurable} } \subseteq L^p(w;X) .
  \end{align*}
  Then we have $\mc{A} \subset L^q(v(\cdot,\omega)w)$ for $\mu$-a.e. $\omega \in \Omega$, since $\mathcal{A}$ is countable and $L^p(w;X)\subseteq L^q(\RR^d \times \Omega, v w)$. Thus $v(\cdot,\omega)w \in L^1_{\loc}(\RR^d)$ and therefore we know that $\mc{A}$ is dense in $L^q(v(\cdot,\omega)w)$ for $\mu$-a.e. $\omega \in \Omega$.  Moreover testing \eqref{eqn:v-extraction} on all $f \in \mc{B}$ we find that
  \begin{equation*}
    \int_{\RR^d} |M\varphi(x)|^q \, v(x,\omega) w(x) \dd x \leq \inc_{X,p,q}([w]_{A_p}) \int_{\RR^d} |\varphi(x)|^q v(x,\omega)w(x) \dd x
  \end{equation*}
for $\mu$-almost all $\omega \in \Omega$ and all $\varphi \in \mc{A}$, again since $\mc{A}$ is countable. So using Proposition \ref{prop:muckenhoupt}\eqref{it:mw4}, we find that
\begin{equation*}
  [v(\cdot,\omega)w]_{A_q} \leq \nrm{M}^q_{\mc{B}(L^q(v(\cdot,\omega)w))} \leq \inc_{X,p,q}([w]_{A_p}), \qquad \text{$\mu$-a.e. $\omega \in \Omega$}
\end{equation*}
as claimed.
\end{proof}

Now we can prove the main extrapolation theorem.

\begin{proof}[Proof of Theorem \ref{thm:pair-extrap-p}]\

  {\bf  Step 1:} $p_0=1$.
  Let $(f,g) \in \mc{F}$ and $w\in A_{p}$. By Proposition \ref{prop:UMDtype} there exists $q > 1$ such that $X$ is $q$-convex.
  Consider a nonnegative function $u \in L^{(p/q)^\prime}(w;(X^q)^*)$ and associate  $v \in L^{(p/q)^\prime}(w;(X^q)^*)$ with $u$ as in Lemma \ref{lem:RdF-algo}.
  Then we have
  \begin{align*}
    \int_{\RR^d} \int_\Omega  f^q u \dd\mu \, w\dd x
    &\leq \int_{\RR^d} \int_\Omega  f^q vw  \dd\mu \dd x \\
    &\leq \int_\Omega  \inc_q([v(\cdot,\omega)w]_{A_q}) \int_{\RR^d}   g^q\, vw \dd x \dd \mu \\
    &\leq \inc_{X,p,q} ( [w]_{A_p} )\int_{\RR^d} \int_\Omega  g^q\, v  \dd\mu \, w \dd x \\
    &\leq \inc_{X,p,q} ( [w]_{A_p} ) \nrm{g^q}_{L^{p/q}(w,X^q)} \nrm{v}_{L^{(p/q)^\prime}(w;(X^q)^*)} \\
    &\leq \inc_{X,p,q} ( [w]_{A_p} ) \nrm{g}_{L^p(w,X)}^q \nrm{u}_{L^{(p/q)^\prime}(w;(X^q)^*)},
  \end{align*}
  using the assumption \eqref{eqn:extrap-assn} and $v(\cdot,\omega)w \in A_q$ for a.e $\omega \in \Omega$ in the second line, and \eqref{eqn:vw-Aq} in the third line.
  Taking the supremum over all normalised $u$ yields \eqref{eqn:extrap-goal}.

  {\bf Step 2:} General $p_0\in (0, \infty)$.
  We argue as in \cite[Corollary 3.14]{CMP11}.	Define a set of pairs $\mc{F}^{p_0}$ by
  \begin{equation*}
    \mc{F}^{p_0} := \{(f^{p_0}, g^{p_0}) : (f,g) \in \mc{F}\}.
  \end{equation*}
  For all $p > p_0$, $w \in A_{p/p_0}$ and $(f^{p_0}, g^{p_0}) \in \mc{F}^{p_0}$, we then have
  \begin{equation*}
    \nrm{f(\cdot,\omega)^{p_0}}_{L^{p/p_0}(w)} \leq \inc_p ([w]_{A_{p/p_0}}) \nrm{g(\cdot,\omega)^{p_0}}_{L^{p/p_0}(w)}, \qquad \mu\text{-a.e. } \omega \in \Omega
  \end{equation*}
  by \eqref{eqn:extrap-assn}.
  Thus we may apply Step 1 to the set $\mc{F}^{p_0}$ and the UMD space $X^{p_0}$, yielding
  \begin{equation*}
    \nrm{f^\prime}_{L^{p/p_0}(w;X^{p_0})} \leq \inc_{X,p,p_0} ( [w]_{A_{p/p_0}} ) \nrm{g^\prime}_{L^{p/p_0}(w;X^{p_0})}
  \end{equation*}
  for all $(f^\prime,g^\prime) \in \mc{F}^{p_0}$ and all $w \in A_{p/p_0}$.
  Since  $(f,g) \in \mc{F}$ if and only if $(f^{p_0}, g^{p_0}) \in \mc{F}^{p_0}$, we get
  \begin{equation*}
    \nrm{g^{p_0}}_{L^{p/p_0}(w;X^{p_0})} \leq \inc_{X,p,p_0} ( [w]_{A_{p/p_0}} ) \nrm{f^{p_0}}_{L^{p/p_0}(w;X^{p_0})}
  \end{equation*}
  for all $(f,g) \in \mc{F}$ and all $w \in A_{p/p_0}$.
  Rearranging this yields \eqref{eqn:extrap-goal} for all $p > p_0$ and all $w \in A_{p/p_0}$.
\end{proof}

It is now easy to prove an extrapolation result for operators (which also implies Theorem \ref{thm:op-extrap-intro} from the introduction).

\begin{cor}\label{cor:op-extrap}
  Fix $p_0 \in (0,\infty)$.
  Suppose $T:\Sigma(\RR^d)\to L^0(\RR^d)$ satisfies
  \begin{equation}\label{eq:assTcor}
    |T(f) - T(g)|\leq |T(f-g)|, \qquad f,g\in \Sigma(\RR^d),
  \end{equation}
  and assume $T$ extends to a bounded operator on $L^p(\RR^d,w)$ for all $p > p_0$ and all Muckenhoupt weights $w \in A_{p/p_0}$ with
  \begin{equation*}
    \nrm{T}_{\mc{B}(L^p(w))} \leq \inc_{p,p_0}([w]_{A_{p/p_0}}).
  \end{equation*}
  Let $X$ be a Banach function space and assume that for all $f \in \Sigma(\RR^d;X)$ the function $\map{\widetilde{T}f}{\RR^d}{X}$, defined by
  \begin{equation*}
    \widetilde{T}f(x,\omega):= \bigl(Tf(\cdot,\omega)\bigr)(x), \qquad x \in \RR^d, \quad \omega \in \Omega,
  \end{equation*}
  is well-defined and strongly measurable.
  If $X^{p_0}\in \UMD$, then $\widetilde{T}$ extends to a bounded operator on $L^p(w;X)$ for all $p > p_0$ and $w \in A_{p/p_0}$, with
  \begin{equation}\label{eqn:extn-norm-est}
    \| \widetilde{T} \|_{\mc{B}(L^p(w;X))} \leq \inc_{X,p,p_0}([w]_{A_{p/p_0}}).
  \end{equation}
\end{cor}

\begin{proof}
  Applying Theorem \ref{thm:pair-extrap-p} with
  \begin{equation}\label{eqn:Ft}
    \mc{F}_T = \cbrace{(|f|,|\widetilde{T}f|) : f\in \Sigma(\RR^d;X)}
  \end{equation}
  yields $\|\widetilde{T} f\|_{L^p(w;X)}\leq \inc_{X,p,p_0} ( [w]_{A_{p/p_0}}) \|f\|_{L^p(w;X)}$ for all $w\in A_{p/p_0}$ and all $f\in \Sigma(\RR^d;X)$.
  Therefore by Lemma \ref{lem:densityextension}, $\widetilde{T}$ extends to a bounded operator on $L^p(w;X)$ which satisfies \eqref{eqn:extn-norm-est}.
\end{proof}

\begin{rmk}
  Theorem \ref{thm:maximalfunctionweighted} plays a central role in the proof of Theorem \ref{thm:pair-extrap-p}, and so it may not be deduced as a consequence of Corollary \ref{cor:op-extrap}, even though this appears possible.
\end{rmk}

\begin{rmk}
  If one omits the condition \eqref{eq:assTcor} in Corollary \ref{cor:op-extrap}, then the proof shows that the estimate
  \begin{equation*}
    \| \widetilde{T}f \|_{L^p(w;X)} \leq \inc_{X,p,p_0}([w]_{A_{p/p_0}}) \|f\|_{L^p(w;X)}, \qquad f\in \Sigma(\RR^d;X)
  \end{equation*}
  still holds for simple functions.
  The condition \eqref{eq:assTcor} was only applied to extend $\widetilde{T}$ to all of $L^p(w;X)$.
  In applications it may be possible to extend $\widetilde{T}$ in some other way.
\end{rmk}

\begin{example}
  Let $X = L^q$ with $q\in [1, \infty)$.
  Then $X^{p_0} = L^{q/p_0}\in \UMD$ if and only if $q\in (p_0, \infty)$.
  If $p_0\geq 1$, this leads to restrictions on the possible values of $q$ to which we can apply the stated extrapolation results.
\end{example}

\begin{rmk}\label{rem:domain}
  In the results above, $\RR^d$ may be replaced by an open subset $\Omega \subset \RR^d$ by standard restriction-extension arguments.
  For example, given a bounded operator $T$ on $L^p(\Omega,w)$, one can define $\overline{T}$ on $L^p(\RR^d,w)$ by $\overline{T} f = E_{\Omega} T (f|_{\Omega})$, where $E_{\Omega}\colon L^p(\Omega)\to  L^p(\RR^d)$ is given by $E_{\Omega} f = f$ on $\Omega$ and $E_{\Omega} f = 0$ on $\RR^d\setminus \Omega$.
  Note that $\|T\| = \|\overline{T}\|$.
  Further extensions to more general metric measure spaces can be made as long as the lattice maximal function is bounded, but this requires further investigation.
\end{rmk}

\section{Fourier multipliers\label{sec:applications}}

The Fourier transform and Fourier multipliers on vector-valued functions are defined similarly to the scalar-valued case.
We use the following normalisation of the Fourier transform:
\begin{equation*}
  \wh{f}(\xi) = \FF f(\xi) := \int_{\RR^d} f(t) e^{-2\pi i t\cdot\xi}\dd t, \qquad f\in L^1(\RR^d;X), \, \xi \in \RR^d.
\end{equation*}
Let $\Sch(\RR^d;X)$ denote the space of $X$-valued Schwartz functions and $\Sch'(\RR^d;X): = \calL(\Sch(\RR^d), X)$ the space of $X$-valued tempered distributions.
For $m\in L^\infty(\RR^d;\calL(X,Y))$, define the \emph{Fourier multiplier} $\map{T_m}{\Sch(\RR^d)\otimes X}{\Sch'(\RR^d;X)}$ by
\begin{equation*}
  T_m f := \FF^{-1}(m \wh{f}).
\end{equation*}
For every $p\in (1, \infty)$ and $w\in A_\infty$, the Schwartz functions $\Sch(\RR^d)$ are dense in $L^p(w)$ (see \cite[Ex.\ 9.4.1]{lG09}), and so $\Sch(\RR^d)\otimes X$ is dense in $L^p(w;X)$.
Thus the $L^p$-boundedness of $T_m$ reduces to the estimate
\begin{equation*}
  \|T_mf\|_{L^p(w;Y)} \lesssim \|f\|_{L^p(w;X)}, \qquad f \in \Sch(\RR^d) \otimes X.
\end{equation*}
A major obstacle in vector-valued Fourier analysis is that the Fourier transform is bounded on $L^2(\RR^d;X)$ if and only if $X$ is isomorphic to a Hilbert space, so proving boundedness of Fourier multipliers on $L^2(\RR^d;X)$ is already difficult.
We refer to \cite{HNVW16} for a detailed treatment of vector-valued Fourier multipliers.

We prove various Fourier multiplier theorems for the real line, which may be transferred to the torus via the following result, which will also be applied to the variational Carleson operator in Section \ref{sec:vcarl}.

\begin{prop}[Transference]\label{prop:transference}
  Let $p\in(1,\infty)$.
  Let $w\in L^1_{\rm loc}(\RR^d)$ be
  $\ZZ^d$-periodic, and let $\overline{w}$ be the associated weight on $\TT^d$.
  Let $m \in L^\infty(\RR^d;\mc{L}(X,Y))$, and suppose every point of $\ZZ^d$ is a Lebesgue point of the function $m(\cdot)x$ for all $x \in X$.
  If $\map{T_m}{L^p(w;X)}{L^p(w;Y)}$ is bounded, then $\map{T_{m|_{\ZZ^d}}}{L^p(\overline{w};X)}{ L^p(\overline{w};Y)}$ is bounded with $\|T_{m|_{\ZZ^d}}\| \leq \|T_m\|$.
\end{prop}

\begin{proof}
  The unweighted version of this result is proved in \cite[Section 5.7a]{HNVW16}, and the proof generalizes directly to the weighted setting.
\end{proof}

We start with a simple extension of scalar Fourier multiplier theory to certain Banach function spaces.

\begin{thm}\label{thm:Fouriermultipliergeneral}
Let $a\in (\frac{d}{2}, d]$ be an integer.
Assume that $m\in L^\infty(\RR^d)$ and that $m\in C^a(\RR^d\setminus\{0\})$ satisfies
\begin{equation}\label{eq:multipliercond}
\sup_{\substack{R>0\\\abs{\alpha} \leq a}} R^{|\alpha|} \Big(R^{-1}\int_{R\leq |\xi|<2R} |D^{\alpha} m(\xi)|^2 \dd \xi\Big)^{1/2}<\infty.
\end{equation}
Let $X$ be a Banach function space with $X^{p_0}\in \UMD$ for some $p_0> d/a$.
Then for every $p\in (\frac{d}{a}, \infty)$ and $w\in A_{\frac{pa}{d}}$, $T_m$ is bounded on $L^p(w;X)$.
\end{thm}

Condition \eqref{eq:multipliercond} is usually called the \emph{H\"ormander--Mihlin condition}.
It holds in particular if $\sup_{\xi\neq 0} |\xi|^{|\alpha|} |D^{\alpha} m(\xi)|<\infty$.

\begin{proof}
Fix $p>d/a$ and let $q\in (d/a,\min\{p,p_0\})$.
By the proof of \cite[Theorem IV.3.9]{GR85} and Theorem \ref{thm:increasingconstant}, for all $w\in A_{p/q}$ we have
\[\|T_m \|_{\mc{L}(L^p(w))} \leq \inc_{m,a,p,q}([w]_{A_{p/q}}).\]
Therefore, by Corollary \ref{cor:op-extrap}, $T_m$ extends to a bounded linear operator on $L^p(w;X)$ for  all $w\in A_{p/q}(\RR^d)$.

Finally, fix $w\in A_{\frac{pa}{d}}(\RR^d)$. Then $w \in A_{p/q}$ for some $q\in (d/a,\min\{p,p_0\})$ by Proposition \ref{prop:muckenhoupt}\eqref{it:mw5}, and thus the required boundedness result for $T_m$ follows.
\end{proof}

\begin{rmk} \
\begin{enumerate}[(i)]
\item
By \cite[Theorem 4]{jR86}, the assumption on $X$ holds for any $\UMD$ Banach function space if $a=d$.
\item Analogous results in the unweighted case also hold for operator-valued multipliers under Fourier type conditions on the Banach spaces (see \cite{GW03, Hyt04}).
\end{enumerate}
\end{rmk}

As another application of Corollary \ref{cor:op-extrap} we prove a multiplier theorem of Coifman--Rubio de Francia--Semmes type \cite{CRS88} (see \cite{sK14} and \cite{Xu96} for weighted extensions in the scalar case), which extends \cite[Theorem A(i)]{sK14} to the vector-valued setting.
In order to state the result we recall the definition of bounded $s$-variation.
Let $\map{m}{\RR}{\CC}$ and $s \in [1,\infty)$.
For each bounded interval $J = [J_-,J_+] \subset \RR$, we say that $m$ has \emph{bounded $s$-variation on $J$} if
\begin{equation*}
  \nrm{m}_{V^s(J)} := \|m\|_\infty +  [m]_{V^s(J)} <\infty,
\end{equation*}
where $[m]_{V^s(J)}^s  := \sup \sum_{i=1}^N |m(t_{i-1}) - m(t_i)|^s $, with supremum taken over all increasing sequences $J_- = t_0 <  \cdots < t_N = J_+$.
Let $\Delta$ be the standard dyadic partition of $\RR \setminus \{0\}$,
\begin{equation*}
  \Delta = \{\pm[2^{k}, 2^{k+1}) : k\in \ZZ\}.
\end{equation*}
We say that $f$ is of {\em bounded $s$-variation uniformly on dyadic intervals} if
\begin{equation*}
  \sup_{J \in \Delta} \nrm{f|_J}_{V^s(J)}<\infty.
\end{equation*}

To prove the following result one uses \cite[Theorem A(i)]{sK14} and the same argument as in the proof of  Theorem \ref{thm:Fouriermultipliergeneral}.
Results for operator-valued multipliers cannot be proved with this method; these are presented in \cite{ALV2}.

\begin{thm}\label{thm:multintrores}
  Let $s \in [1,2]$, and let $X$ be a Banach function space with $X^s \in \UMD$.
  Then for all $\map{m}{\RR}{\CC}$ of bounded $s$-variation uniformly on dyadic intervals, the Fourier multiplier $T_m$ extends boundedly to $L^p(w;X)$ for all $p > s$ and $w \in A_{p/s}$.
\end{thm}

By duality one obtains a similar result for $X$ such that $(X^*)^s \in \UMD$.
Precise details are left to the reader.

\section{Vector-valued variational Carleson operators}\label{sec:vcarl}

Let $X$ be a Banach function space.
For $r < \infty$, the $X$-valued \emph{variational Carleson operator} $C_r$ is defined on $f \in \Sch(\RR;X)$ by
\begin{equation*}
  C_r f(t) := \sup_{N \in \NN} \sup_{\xi_0 < \cdots < \xi_N} \has{ \sum_{j=1}^N \abss{ \int_{\xi_{j-1}}^{\xi_j} \wh{f}(\xi) e^{it \xi} \, d \xi }^r }^{1/r} \qquad t \in \RR.
\end{equation*}
For each $t \in \RR$, $C_r f(t)$ is the $r$-variation of the partial inverse Fourier transform
\begin{equation*}
  \eta \mapsto \int_{-\infty}^{\eta} \wh{f}(\xi) e^{it\xi} \dd\xi.
\end{equation*}
This is a strengthening of the more classical Carleson operator
\begin{equation*}
  Cf(t) := \sup_{N < \infty} \abss{ \int_{-\infty}^N \wh{f}(\xi) e^{it\xi} \dd\xi },
\end{equation*}
which formally corresponds to the operator $C_\infty$.
Versions of these operators on the one-dimensional torus $\TT = \RR / \ZZ$ can be easily defined, and we will denote these by $C_r^{\TT}$.

Boundedness of $C_\infty^{\TT}$ on $L^p$ for all $p \in (1,\infty)$ is the celebrated Carleson--Hunt theorem (see for example \cite[Chapter 11]{lG09}); a consequence of this boundedness is the pointwise convergence of Fourier series $f(t) = \sum_{k \in \ZZ} \wh{f}(k) e^{-itk}$ for $f \in L^p(\TT)$ and a.e. $t \in \TT$ (an analogous result holds for Fourier integrals, replacing $\TT$ with $\RR$).
This is a qualitative result: the Fourier series (or integral) of an $L^p$ function is guaranteed to converge pointwise a.e., but no information on the rate of convergence is obtained.
Using the extrapolation result which inspired our Theorem \ref{thm:pair-extrap-p}, Rubio de Francia proved that $C_\infty$ is bounded on $L^p(\RR;X)$ for all $\UMD$ Banach lattices \cite[p. 219]{jR86}. See also \cite[Corollary 3.5]{RRT86} for this result on $\UMD$ Banach spaces with an unconditional basis, and more recently \cite{HL13} on `intermediate' $\UMD$ spaces, including the Schatten classes $\Sch^p$.

The $r$-variation of partial inverse Fourier integrals provides {\em quantitative} information on the rate of convergence of Fourier integrals, which motivates investigation of the  boundedness of $C_r$ on $L^p(\RR;X)$ (of course the same holds for Fourier series).
In the scalar case the following result holds; the unweighted case is in \cite[Theorem 1.1]{OSTTW12}, and the weighted case is in \cite[Theorem 2(ii)]{dPDU16} (see also \cite{DL12} for related estimates).

\begin{thm}
  Suppose $r \in (2,\infty)$.
  Then for all $p \in (r^\prime,\infty)$ and $w \in A_{p/r^\prime}(\RR)$, $C_r$ is bounded on $L^p(w)$ with
  \begin{equation*}
    \nrm{C_r}_{\mc{B}(L^p(w))} \leq \inc_{p,r}([w]_{A_{p/r^\prime}}).
  \end{equation*}
\end{thm}

This is precisely the kind of estimate that we can extrapolate via Corollary \ref{cor:op-extrap}.
The result is the following theorem, which is new even in the unweighted case.

\begin{thm}\label{thm:carlesonR}
  Suppose $r \in (2,\infty)$, and let $X$ be a Banach function space with $X^{r^\prime} \in \UMD$.
  Then for all $p \in (r^\prime,\infty)$ and $w \in A_{p/q^\prime}$, $C_r$ is bounded on $L^p(w;X)$ with
  \begin{equation*}
    \nrm{C_r}_{\mc{B}(L^p(w;X))} \leq \inc_{X,p,r}([w]_{A_{p/r^\prime}}).
  \end{equation*}
\end{thm}

Using the transference result of Proposition \ref{prop:transference}, we can deduce an analogous result for $C_r^{\TT}$.

\begin{cor}\label{cor:carlesonT}
  Suppose $r \in (2,\infty)$, and let $X$ be a Banach function space with $X^{r^\prime} \in \UMD$.
  Then for all $p \in (r^\prime,\infty)$ and $w \in A_{p/q^\prime}$, $C_r$ is bounded on $L^p(w;X)$ with
  \begin{equation*}
    \nrm*{C_r^{\TT}}_{\mc{B}(L^p(w;X))} \leq \inc_{X,p,r}([w]_{A_{p/r^\prime}}).
  \end{equation*}
\end{cor}

\begin{proof}
  Fix $N,M \in \NN$, and let $\overline{w}$ be the $\ZZ$-periodic extension of $w$ to $\RR$.
  Let
  \begin{equation*}
    \ell^\infty_{M,N} := \ell^\infty(\{-M,\ldots,M\}^N) \quad \text{and} \quad \ell^r_{N-1} := \ell^r(\{1,\ldots,N-1\}).
  \end{equation*}
  Define a bounded operator-valued function
  \begin{equation*}
    m \in L^\infty\ha*{\RR;\mc{L}\ha*{X,X(\ell^\infty_{M,N}(\ell^r_{N-1}))}}
  \end{equation*}
  as follows: for $t \in \RR$, $x \in X$, $\mb{n} = (n_1,\ldots,n_N) \in \{-M,\ldots,M\}^{N}$, and $1 \leq j \leq N-1$, define
  \begin{equation*}
    m(t)x(\mb{n},j) =\begin{cases}
      \ind_{{(n_{j}-\frac{1}{2}, n_{j+1} - \frac{1}{2})}}(t)x &\text{if } n_1\leq\cdots\leq n_N, \\
      0 &\text{otherwise.}
    \end{cases}
  \end{equation*}
  By combining Theorem \ref{thm:carlesonR} and Proposition \ref{prop:transference} we obtain
  \begin{equation*}
    \nrm{T_{m|_{\ZZ}}} \leq \nrm{T_{m}} \leq \nrm{C_r}_{\mc{B}(L^p(\overline{w};X))} \leq \inc_{X,p,r}([\overline{w}]_{A_{p/r^\prime}}) = \inc_{X,p,r}([w]_{A_{p/r^\prime}}),
  \end{equation*}
  which implies for $f \in L^p(w;X)$ that
  \begin{equation*}
    \nrms{\sup_{-M \leq n_1 \leq\cdots\leq n_N \leq M} \has{\sum_{j=1}^{N-1}\abss{\sum_{k=n_{j}}^{n_{j+1}-1}\hat{f}(k)\ee^{itk}}^r}^{1/r}}_{L^p(w;X)} \leq \inc_{X,p,r}([w]_{A_{\frac{p}{r^\prime}}})\nrm{f}_{L^p(w;X)}
  \end{equation*}
  with $\inc_{X,p,r}$ independent of $M$ and $N$.
  Two applications of the monotone convergence theorem yields the desired result.
\end{proof}

\section{Estimates of Littlewood--Paley--Rubio de Francia type\label{sec:LPR}}

Recall the discussion of the operators $S_I$ and $\mc{S}_{\mc{I},q}$ from the introduction.
In this section we apply Corollary \ref{cor:op-extrap} to the operators $\mc{S}_{\mc{I},q}$.
First we consider the operator $\mc{S}_{\Delta,2}$, where $\Delta := \{ \pm[2^k, 2^{k+1}), k \in \ZZ\}$ is the standard dyadic partition of $\RR \setminus \{0\}$.
Corollary \ref{cor:op-extrap} yields a direct proof of the classical Littlewood--Paley estimate in UMD Banach function spaces.

\begin{prop}\label{prop:LP-UMD}
	Let $X$ be a $\UMD$ Banach function space, $p \in (1,\infty)$, and $w \in A_p$.
	Then for all $f\in L^p(w;X)$,
	\begin{equation*}
	\inc_{X,p}([w]_{A_p})^{-1}
 \nrm{f}_{L^p(w;X)}\leq \nrm{\mc{S}_{\Delta,2}(f)}_{L^p(w;X)}\leq \inc_{X,p}([w]_{A_p})
\nrm{f}_{L^p(w;X)}.
	\end{equation*}
      \end{prop}

\begin{proof}
  In the scalar case the result was obtained in \cite[Theorem 1]{dK80}, using Theorem \ref{thm:increasingconstant} for the monotonicity in $[w]_{A_p}$.
  Therefore the estimate
  \begin{equation*}
    \nrm{\mc{S}_{\Delta,2}(f)}_{L^p(w;X)}\leq \inc_{X,p}([w]_{A_p}) \nrm{f}_{L^p(w;X)}
  \end{equation*}
  follows from Corollary \ref{cor:op-extrap}.
  The converse estimate may be proved using a duality argument or Theorem \ref{thm:pair-extrap-p}.
\end{proof}

\begin{rmk}
  Theorem \ref{prop:LP-UMD} actually holds for all $\UMD$ Banach spaces, and was proved in \cite{Bou86,Zim89} in the unweighted case and in \cite{FHL17} in the weighted case.
  Here the $\ell^2$-sum in $\nrm{\mc{S}_{\Delta,2}(f)}_{L^p(w;X)}$ must be replaced by a suitable Radem\-acher sum.
\end{rmk}

Next we establish weighted Littlewood--Paley--Rubio de Francia estimates for Banach function spaces with $\UMD$ concavifications (Theorem \ref{thm:lpr-main-intro} in the introduction).
The unweighted case with $q=2$ was first proved in \cite{PSX12}, but we do not use this result in our proof.

\begin{thm}\label{thm:LPR-main}
  Suppose that $q \in [2,\infty)$ and let $X$ be a Banach function space with $X^{q'} \in \UMD$.
  Then for all collections $\mc{I}$ of mutually disjoint intervals, all $p > q^\prime$, $w \in A_{p/q^\prime}$, and $f\in L^p(w;X)$,
  \begin{equation*}
    \|\mc{S}_{\mc{I},q} (f)\|_{L^p(w;X)} \leq \inc_{X,p,q}([w]_{A_{p/q^\prime}}) \nrm{f}_{L^p(w;X)}.
  \end{equation*}
\end{thm}

In the scalar case there is also a weak-type estimate for $p=q'$ and $w\in A_1$.
The strong-type estimate seems to remain an open problem (see \cite[(6.4)]{jR85}).

\begin{proof}
  The scalar case of this result is proved in \cite[Theorem 6.1]{jR85} for $q = 2$, and \cite[Theorem B]{sK14} for $q > 2$.
  Monotonicity in  $[w]_{A_{p/q^\prime}}$ is contained in \cite{sK14} for $q > 2$, and can be deduced from \cite{jR85} combined with Theorem \ref{thm:increasingconstant} when $q=2$.
  Thus the result follows immediately from Corollary \ref{cor:op-extrap}.
\end{proof}

\begin{rmk}\ \label{rmk:banachlattice-HS}
As observed in \cite{PSX12}, Theorem \ref{thm:LPR-main} still holds under the assumption that $X$ is a Banach lattice rather than a Banach function space (see \cite[Theorem 1.b.14]{LT79}).
\end{rmk}

When $q=2$, the estimate in Theorem \ref{thm:LPR-main} can be used to obtain extensions of the Marcinkiewicz multiplier theorem.
This is done in \cite[Theorem 2.3]{HP06}.
For $q>2$ a slight variation will be needed to make this work.
The following estimate, which combines Proposition \ref{prop:LP-UMD} and Theorem \ref{thm:LPR-main}, is a key ingredient in the Fourier multiplier theory developed in \cite{ALV2}.

\begin{thm}\label{thm:LPRmod-lat}
  Suppose $q\in [2, \infty)$ and let $X$ be a Banach function space such that $X^{q'}\in \UMD$.
  Let $\mc{I}$ be a collection of mutually disjoint intervals in $\RR$, and for all $J \in \Delta$ let
  $\mc{I}^J := \{I \in \mc{I} : I \subset J\}$.
  Then for all $p > q'$, all $w \in A_{p/q'}$ and all $f\in L^p(w;X)$,
\begin{equation*}
\Big\|\Big(\sum_{J\in \Delta} |\mc{S}_{\mc{I^J},q}(f)|^2 \Big)^{1/2}\Big\|_{L^p(w;X)} \leq \inc_{X,p,q}([w]_{A_{p/q'}}) \nrm{f}_{L^p(w;X)}.
\end{equation*}
\end{thm}

\begin{proof}
  If $q = 2$ this follows from Theorem \ref{thm:LPR-main}, so we need only consider $q > 2$.
  By Corollary \ref{cor:op-extrap} it suffices to consider $X = \CC$, and by Theorem \ref{thm:extrapolation} (scalar-valued extrapolation) it suffices to take $p=2$.
  Now estimate
  \begin{align*}
\Big\|\Big(\sum_{J\in \Delta} |\mc{S}_{\mc{I}^J,q}(f)|^2 \Big)^{1/2}\Big\|_{L^2(w)}  &= \has{ \sum_{J \in \Delta} \nrms{ \big( \sum_{I \in \mc{I}^J} |S_I S_J f|^{q} \big)^{1/q} }_{L^2(w)}^2 }^{1/2} \\
    &\leq \inc_{q}([w]_{A_{2/q'}}) \has{ \sum_{J \in \Delta} \nrm{S_J f}_{L^2(w)}^2 }^{1/2} \\
    &\leq \inc_{q}([w]_{A_{2/q'}}) \nrm{ \mc{S}_{\Delta,2}f }_{L^2(w)} \\
    &\leq \inc_{q} ([w]_{A_{2/q'}}) \nrm{f}_{L^2(w)}
  \end{align*}
  using the scalar case of Theorem \ref{thm:LPR-main} (noting that $q^\prime < 2$) in the third line, and Proposition \ref{prop:LP-UMD} in the last line.
\end{proof}

If $X$ is a Hilbert space, then one cannot apply Theorem \ref{thm:LPR-main} with $q=2$.
Instead, the following modification of Theorem \ref{thm:LPR-main} holds.
\begin{prop}\label{prop:HScaseq2}
  Let $X$ be a Hilbert space, and let $\mc{I}$ be a collection of mutually disjoint intervals in $\RR$.
  Then for all $p > 2$, $w \in A_{p/2}$, and $f\in L^p(w;X)$,
  \begin{equation*}
    \Big\|\big( \sum_{I \in \mc{I}} \|S_I f\|^2_X \big)^{1/2}\Big\|_{L^p(w)} \leq \inc_{p}([w]_{A_{p/2}}) \nrm{f}_{L^p(w;X)}.
  \end{equation*}
\end{prop}

\begin{proof}
  To prove this it suffices to consider $X = \ell^2$ (by restriction to a separable Hilbert space, see \cite[Theorem 1.1.20]{HNVW16}).
  Now the result follows from Fubini's theorem, the result in the scalar-valued case, and a randomisation argument.

  Let $(\varepsilon_{I})_{I\in \mc{I}}$ and $(r_n)_{n\geq 1}$ be a Rademacher sequences on probability spaces $\Omega_{\varepsilon}$ and $\Omega_{r}$ respectively.
  Then writing $F = \sum_{n\geq 1}r_n f_n\in L^p(w;L^p(\Omega_{r}))$, where $f = (f_n)_{n\geq 1}$, it follows from Fubini's theorem and Khintchine's inequality (see \cite[Corollary 3.3.24]{HNVW16}) that
\begin{equation*}
\Big\|\big( \sum_{I \in \mc{I}} \|S_I f\|^2_{\ell^2} \big)^{1/2}\Big\|_{L^p(w)} \eqsim_p  \Big\| \sum_{I \in \mc{I}} \varepsilon_I S_I F\Big\|_{L^p(\Omega_{r};L^p(w;L^p(\Omega_{\varepsilon})))}.
\end{equation*}
Now we can argue pointwise in $\Omega_{r}$.
By Khintchine's inequality and the scalar case of the Littlewood--Paley--Rubio de Francia theorem \cite[Theorem 6.1]{jR85}, we obtain
\[\Big\|\sum_{I \in \mc{I}} \varepsilon_I S_I F\Big\|_{L^p(w;L^p(\Omega_{\varepsilon}))}  \eqsim_p \Big\|\big( \sum_{I \in \mc{I}} |S_I F|^2_X \big)^{1/2}\Big\|_{L^p(w)}
\leq \inc([w]_{A_{p/2}}) \|F\|_{L^p(w)}.\]
The result now follows by taking $L^p(\Omega_{r})$-norms and applying Khintchine's inequality once more.
\end{proof}

\begin{rmk}
If $X$ is a Hilbert space, $\mc{I}$ a collection of mutually disjoint intervals in $\RR$ and $q>2$, then for all $p > q'$, $w \in A_{p/q'}$ and $f\in L^p(w;X)$, we have
\begin{align*}
	\Big\|\big( \sum_{I \in \mc{I}} \|S_I f\|^q_X \big)^{1/q}\Big\|_{L^p(w)} &\leq \inc_{p,q}([w]_{A_{p/q'}}) \nrm{f}_{L^p(w;X)}\\
\nrms{\Bigl(\sum_{J\in \Delta} \Bigl( \sum_{I \in \mc{I}^J} \nrm{S_I f}_X^q \Bigr)^{2/q} \Bigr)^{1/2}}_{L^p(w)} &\leq \inc_{p,q}([w]_{A_{p/q'}}) \nrm{f}_{L^p(w;X)}.
\end{align*}
These estimates are weaker than Theorem \ref{thm:LPR-main} and Theorem \ref{thm:LPRmod-lat}.
To prove the first estimate it is enough to consider $X = \ell^2$.
In this case \[\Big\|\big( \sum_{I \in \mc{I}} \|S_I f\|^q_{\ell^2} \big)^{1/q}\Big\|_{L^p(w)}  \leq \|\mc{S}_{\mc{I},q} f\|_{L^p(w;\ell^2)}\] by Minkowski's inequality, so the result follows from Theorem \ref{thm:LPR-main}. The second estimate is proved similarly.
\end{rmk}

\appendix

\section{Monotone dependence on Muckenhoupt characteristics}\label{sec:weight dependence}
For scalar-valued extrapolation (Theorem \ref{thm:extrapolation}) one needs an estimate of the form
\begin{equation}\label{eqn:wt-mono}
  \nrm{ f}_{L^{p}(w)} \leq \inc([w]_{A_{p/p_0}}) \nrm{ g}_{L^{p}(w)}
\end{equation}
for all $w \in A_{p/p_0}$, where $\phi\colon [1,\infty) \to [1,\infty)$ is a nondecreasing function independent of $w$; this is often overlooked in the literature.
In applications it is often easily checked that a weighted estimate is dependent on the Muckenhoupt characteristic $[w]_{A_{p/p_0}}$, and not on any other information coming from $w$, see for example \cite{HHH03,sK14}.
However, checking that this dependence is nondecreasing in $[w]_{A_{p/p_0}}$ can be tricky (see for example \cite[Theorem 3.10]{GLV15}).
Moreover, this monotonicity is usually not explicitly stated in the literature.

In this appendix we show that the monotonicity condition in \eqref{eqn:wt-mono} is redundant when working with a set of pairs of nonnegative functions: an estimate depending on $[w]_{A_{p/p_0}}$ with no monotonicity assumption implies the estimate \eqref{eqn:wt-mono}.

\begin{thm}\label{thm:increasingconstant}
  Fix $p_0 \in (0, \infty)$ and $p \in (p_0,\infty)$.
  Let $\mc{F} \subset L_+^0(\RR^d) \times L_+^0(\RR^d)$ and suppose that there exists a function $C\colon[1,\infty) \to [1,\infty)$ such that for all $(f,g) \in \mc{F}$ and $w \in A_{p/p_0}$ we have
  \begin{equation*}
    \nrm{f}_{L^p(w)} \leq C([w]_{A_{p/p_0}}) \nrm{g}_{L^p(w)}.
  \end{equation*}
  Then there exists a nondecreasing function $\inc\colon[1,\infty) \to [1,\infty)$ such that $\inc(t) \leq C(t)$ for all $t \in [1,\infty)$ and such that for all $(f,g) \in \mc{F}$ and $w \in A_{p/p_0}$
  \begin{equation}\label{eq:increasing}
    \nrm{f}_{L^p(w)} \leq \inc([w]_{A_{p/p_0}}) \nrm{g}_{L^p(w)}.
  \end{equation}
\end{thm}

\begin{proof}
  By rescaling $f$ and $g$ we may take $p_0 =1$. Without loss of generality we may assume that $f,g \in L^p(w)$ for all $(f,g) \in \mc{F}$.
  Define $\inc:[1,\infty) \to [1,\infty)$ by
  \begin{equation*}
    \phi(t) := \sup \cbraces{\frac{\nrm{f}_{L^p(w)}}{\nrm{g}_{L^p(w)}}:(f,g) \in \mc{F}, w \in A_p, [w]_{A_p} =t}.
  \end{equation*}
  Then $\inc(t) \leq C(t)$ for all $t \in [1,\infty)$, and \eqref{eq:increasing} holds.

  We will show that $\inc$ is nondecreasing.
  Let $1 \leq t < s < \infty$ and $\varepsilon >0$.
  Fix $w \in A_p$ with $[w]_{A_p}=t$ and $(f,g) \in \mc{F}$ such that
  \begin{equation*}
    \nrm{f}_{L^p(w)} \geq \bigl( \inc([w]_{A_p})-\varepsilon\bigr)\nrm{g}_{L^p(w)},
  \end{equation*}
  and fix a ball $B_0 \subset \RR^d$ such that
  \begin{equation}\label{eq:B0assumptions}
    \nrm{f\ind_{B_0}}_{L^p(w)} \leq \varepsilon \nrm{g}_{L^p(w)} \quad \text{and} \quad \nrm{g\ind_{B_0}}_{L^p(w)} \leq  \frac{\varepsilon}{2{s}^{\frac1p}} \nrm{g}_{L^p(w)}.
  \end{equation}
  Divide $B_0$ into two sets $B_0^+$ and $B_0^-$ such that $\abs{B_0^+} = \abs{B_0^-} = \abs{B_0}/2$ and $w(x) > w(y)$ for all $x \in B_0^+$ and $y \in B_0^-$.
  For any $\sigma \in [1,\infty)$ we define a weight
  \begin{equation*}
    w_\sigma(x) := \begin{cases}
      \sigma \cdot w(x) &\text{if }x \in B_0^+\\
      w(x) &\text{if }x \in B_0^-,
    \end{cases}
  \end{equation*}
  and for $B \subset \RR^d$ define a function $f_B \colon [1,\infty) \to [1,\infty)$ by
  \begin{equation*}
    f_B(\sigma) := \frac{1}{\abs{B}}\int_B w_\sigma(x) \dd x \, \has{\frac{1}{\abs{B}} \int_B w_\sigma^{-1/(p-1)} \dd x}^{p-1}.
  \end{equation*}
  Then $f_B$ is of the form
  \begin{equation*}
    f_B(\sigma) = (\alpha_0 + \alpha_+ \cdot \sigma)\ha*{\beta_0 +  \beta_+ \cdot \sigma^{-\frac{1}{p-1}} }^{p-1}
  \end{equation*}
  with $\alpha_- ,\alpha_+,\beta_- ,\beta_+$ constants depending on $B$ which satisfy
  \begin{equation*}
    \alpha_- <  \alpha_+, \qquad \beta_- >\beta_+, \qquad(\alpha_-+\alpha_+)(\beta_-+\beta_+)^{p-1} \leq [w]_{A_p}.
  \end{equation*}
  So if we restrict to $[1,2^p s]$ we know that $f_B \in C^1([1,2^p s])$ with norm independent of $B$.

  For each $n \in \NN$ define a function
  \begin{equation*}
    f_n := \sup_{B \in \mc{B}_n} f_B
  \end{equation*}
  on $[1,2^p t]$, where each $\mc{B}_n$ is a finite collection of balls in $\RR^d$, such that $\mc{B}_n \subset \mc{B}_{n+1}$ and $\bigcup_{n=1}^\infty \mc{B}_n$ contains all balls in $\RR^d$ with rational centre and radius.
  Then the sequence $(f_n)_{n=1}^\infty$ is nondecreasing and bounded, so it converges pointwise to some function $f$.
  Restricting to $[1,2^p s]$, we also have that the sequence $(f_n)_{n=1}^\infty$ is equicontinuous, so by the Arzel\`a--Ascoli theorem we know that $f$ is continuous on $[1,2^p s]$.
  By a density argument we get that
  \begin{equation*}
    f(\sigma) = \sup_{\substack{B\subset\RR^d\\B \text{ rational} }} f_B(\sigma) =  \sup_{B\subset \RR^d} f_B(\sigma) = [w_\sigma]_{A_p}.
  \end{equation*}
  Since $f(1) = [w]_{A_p} = t$ and
  \begin{equation*}
    f(2^ps) \geq \frac{1}{\abs{B_0}}\int_{B_0^+} 2^ps w(x) \dd x \, \has{\frac{1}{\abs{B_0}}\int_{B_0^-} w(x)^{-\frac{1}{p-1}}\dd x}^{p-1}  \geq  \frac{f_{B_0}(1)}{2^p}   2^p s \geq s,
  \end{equation*}
  there exists $\sigma \in [1,2^p s]$ such that $s = f(\sigma) = [w_\sigma]_{A_p}$.

  Now by construction and \eqref{eq:B0assumptions} we  have
  \begin{equation*}
    \nrm{g\ind_{B_0}}_{L^p(w_\sigma)} \leq \sigma^{1/p} \nrm{g\ind_{B_0}}_{L^p(\RR^d,w)} \leq \varepsilon\nrm{g}_{L^p(\RR^d,w)}.
  \end{equation*}
  Combining this with \eqref{eq:B0assumptions} and the triangle inequality yields
  \begin{align*}
    \nrm{f}_{L^p(w_s)} &\geq \nrm{f\ind_{B_0^c}}_{L^p(w)}
                         +\nrm{f\ind_{B_0}}_{L^p(w)}-\nrm{f\ind_{B_0}}_{L^p(w)}\\
                       &\geq \nrm{f}_{L^p(w)} - \nrm{f\ind_{B_0}}_{L^p(w)}\\
                       &\geq (\inc(t) - 2\varepsilon)\nrm{g}_{L^p(w)}\\
                       &\geq  (\inc(t) - 2\varepsilon) \has{\nrm{g}_{L^p(w_s)} -\nrm{g\ind_{B_0}}_{L^p(w_s)}}\\
                       &\geq  (\inc(t) - 2\varepsilon)(1 - \varepsilon) \nrm{g}_{L^p(w_s)}.
  \end{align*}
  Thus $\inc(s) \geq (\inc(t) - 2\varepsilon)(1 - \varepsilon)$, and since $\varepsilon>0$ was arbitrary this implies $\inc(s) \geq \inc(t)$, so  $\inc$ is nondecreasing.
\end{proof}

\begin{rmk}
  The proof of Theorem \ref{thm:increasingconstant} can be adapted to allow for $p/p_0 = 1$, in which case we need to deal with the $A_1$-characteristic.
\end{rmk}

Theorem \ref{thm:increasingconstant} implies a result of the same type for vector-valued extrapolation (Theorem \ref{thm:pair-extrap-p}).

\begin{cor}\label{cor:increasingconstantvector}
  Fix $p_0 \in (0, \infty)$, $p \in (p_0,\infty)$. and suppose that $X$ is a Banach function space over a measure space $(\Omega,\mu)$.
  Let $\mc{F} \subset L_+^0(\RR^d;X) \times L_+^0(\RR^d;X)$ and suppose that there exists a function $C\colon[1,\infty) \to [1,\infty)$ such that for all $(f,g) \in \mc{F}$ and $w \in A_{p/p_0}$ we have
  \begin{equation}\label{eq:increasingvector}
    \nrm{f(\cdot,\omega)}_{L^p(w)} \leq C([w]_{A_{p/p_0}}) \nrm{g(\cdot,\omega)}_{L^p(w)}, \qquad \mu\text{-a.e } \omega \in \Omega.
  \end{equation}
  Then there exists an nondecreasing function $\inc\colon[1,\infty) \to [1,\infty)$ such that $\inc(t) \leq C(t)$ for all $t \in [1,\infty)$ and
  \begin{equation*}
    \nrm{f(\cdot,\omega)}_{L^p(w)} \leq \inc([w]_{A_{p/p_0}}) \nrm{g(\cdot,\omega)}_{L^p(w)}, \qquad \mu\text{-a.e } \omega \in \Omega.
  \end{equation*}
\end{cor}

\begin{proof}
  Fix $\Omega_0$ such that \eqref{eq:increasingvector} holds for all $\omega \in \Omega_0$.
  Using Theorem \ref{thm:increasingconstant} for $\omega \in \Omega_0$, we can find $\inc_\omega\colon[1,\infty) \to [1,\infty)$ such that $\phi_\omega(t) \leq C(t)$ for all $t \in [1,\infty)$ and
  \begin{equation*}
    \nrm{f(\cdot,\omega)}_{L^p(w)} \leq \inc_\omega([w]_{A_{p/p_0}}) \nrm{g(\cdot,\omega)}_{L^p(w)}.
  \end{equation*}
  Setting $\phi(t) := \sup_{\omega \in \Omega_0} \phi_\omega(t) \leq C(t)$ proves the corollary.
\end{proof}

\bibliographystyle{plain}

\bibliography{literature}

\end{document}